\documentclass[12pt]{amsart}

\usepackage{times,amsfonts,amsmath,amstext,amsbsy,amssymb,
  amsopn,amsthm,upref,eucal, amscd}
\usepackage[T1]{fontenc}

\newtheorem{theorem}{Theorem}[section]
\newtheorem{lemma}[theorem]{Lemma}
\newtheorem{corollary}[theorem]{Corollary}

\numberwithin{equation}{section}

\theoremstyle{definition}
\newtheorem{definition}[theorem]{Definition}

\newcommand{\field}[1]{\mathbb{#1}}
\newcommand{\Proj}{\field{P}}
\newcommand{\R}{\field{R}}
\newcommand{\N}{\field{N}}
\newcommand{\C}{\field{C}}

\newcommand{\Cal}{\mathcal}

\begin{document}
\title[A Teichm\"uller disk with degenerate Kontsevich-Zorich spectrum]{An example of a Teichm\"uller disk in genus 4 \\with degenerate Kontsevich-Zorich spectrum}
\author{Giovanni Forni}
\address{Giovanni Forni: Department of Mathematics, University of Maryland, College Park, MD 20742-4015, USA}
\email{gforni@math.umd.edu.}
\author{Carlos Matheus}
\address{Carlos Matheus: Coll\`ege de France, 3, Rue d'Ulm, Paris, CEDEX 05, France}
\email{matheus@impa.br.}
\date{September 30, 2008.}
\begin{abstract}
We construct an orientable holomorphic quadratic differential on a Riemann surface of genus $4$ 
whose $SL(2,\mathbb{R})$-orbit is closed and has a highly degenerate Kontsevich-Zorich spectrum.
This example is related to a previous similar construction in genus $3$ by the first author. 
\end{abstract}
\maketitle

\section{Introduction}\label{intro}
The goal of this note is the construction of an orientable quadratic differential $q$ on a Riemann surface $M$ of genus 4 such that the Kontsevich-Zorich cocycle along the $SL(2,\mathbb{R})$-orbit of $(M,q)$ has a {\it totally degenerate }Kontsevich-Zorich spectrum,in the sense that all non-trivial Lyapunov exponents of the Kontsevich-Zorich cocycle are equall to zero. A similar example in genus $3$ appeared in \cite{ForniSurvey}. Our example is given by a non-primitive Veech surface in the stratum $\Cal H(2,2,2)$ of abelian differentials with $3$ double zeros. Unpublished work of M. M\"oller implies that there are very few examples of this kind and only for abelian differentials on surfaces of genus $g\leq 5$. It is likely that the Forni's example in genus $3$ \cite{ForniSurvey} and the example presented here are the only examples of Veech surfaces with vanishing non-trivial Konstevich-Zorich exponents \cite{Moeller}.

In order to explain more precisely our results, let us recall the definition of the Teichm\"uller flow 
and the Kontsevich-Zorich cocycle.

Given a surface $M$, we denote by $\text{ \rm Diff}^+(M)$ the group of orientation-preserving diffeomorphisms of $M$ and $\text{ \rm Diff}_0^+(M)$ the connected component of the identity in $\text{ \rm Diff}^+(M)$ (i.e., $\text{ \rm Diff}_0^+(M)$ is the subset of diffeomorphisms in $\text{ \rm Diff}^+(M)$ which are isotopic to the identity).

\begin{definition}Let $Q_g$ be the \emph{Teichm\"uller} space of holomorphic quadratic differentials on a surface of genus $g\geq1$:
\begin{center}
$Q_g=\{\textrm{holomorphic quadratic differentials}\}/\text{ \rm Diff}_0^+(M)$.
\end{center}
Also, let $\mathcal{M}_g$ be the \emph{moduli} space of holomorphic quadratic differentials on a surface of genus $g\geq 1$:
\begin{center}
$\mathcal{M}_g=\{\textrm{holomorphic quadratic differentials}\}/\Gamma_g,$
\end{center}
where $\Gamma_g$ is the \emph{mapping class group} $\Gamma_g:=\text{ \rm Diff}^+(M)/\text{ \rm Diff}_0^+(M)$.
\end{definition}

Moreover, we note that the group $SL(2,\mathbb{R})$ acts naturally on $\mathcal{M}_g$ by linear transformations on the pairs of real-valued $1$-forms $\left(\textrm{Re}(q^{1/2}),\textrm{Im}(q^{1/2})\right)$. The \emph{Teichm\"uller} (geodesic) flow $G_t$ is given by the action of the diagonal subgroup $\textrm{diag}(e^{t},e^{-t})$ of $SL(2,\mathbb{R})$ on $\mathcal{M}_g$.

\smallskip
For later reference, we recall some of the main structures of the Teichm\"uller space $Q_g$ and the moduli space $\mathcal{M}_g$:

\begin{itemize}
\item $\mathcal{M}_g$ and $Q_g$ are stratified into analytic spaces $\mathcal{M}_\kappa$ and $Q_\kappa$ obtained by fixing the multiplicities $\kappa=(k_1,\dots,k_\sigma)$ of the zeroes $\{p_1,\dots,p_\sigma\}$ of the quadratic differentials (here $\sum k_i=4g-4$);
\item the total area function $A:\mathcal{M}_g\to\mathbb{R}^+$, $A(q) = \int_M |q|$ is $SL(2,\mathbb{R})$-invariant so that the unit bundle $\mathcal{M}_g^{(1)}:=A^{-1}(1)$ and its strata $\mathcal{M}_\kappa^{(1)}:= \mathcal{M}_g\cap \mathcal{M}_g^{(1)}$ are $SL(2,\mathbb{R})$-invariant (and, \emph{a fortiori}, $G_t$-invariant);
\item a stratum $\mathcal{M}_\kappa$ of {\it orientable quadratic differentials }(that is, quadratic differentials obtained as squares of holomorphic 1-forms) has a locally affine structure modeled on the cohomology $H^1(M,\Sigma_\kappa,\mathbb{C})$, relative to the zero set $\Sigma_\kappa:=\{p_1,\dots,p_\sigma\}$, with local charts are given by the period map $q\mapsto [q^{1/2}]\in H^1(M,\Sigma_\kappa,\mathbb{C})$;
\item the Lebesgue measure on the Euclidean space $H^1(M,\Sigma_\kappa,\mathbb{C})$, appropriately normalized, induces an absolutely continuous $SL(2,\mathbb{R})$-invariant measure $\mu_\kappa$ on $\mathcal{M}_\kappa$ such that the conditional measure $\mu_\kappa^{(1)}$ induced on $\mathcal{M}_\kappa^{(1)}$ is $SL(2,\mathbb{R})$-invariant (and hence $G_t$-invariant).
\end{itemize}

Once we get the existence of a good invariant measure $\mu_\kappa^{(1)}$ for the Teichm\"uller flow, it is natural to ask whether $\mu_\kappa^{(1)}$ has finite mass and/or $\mu_\kappa^{(1)}$ is ergodic with respect to the Teichm\"uller dynamics. In this direction, Veech~\cite{Veech2} showed that the strata are \emph{not} always connected. More recently, Kontsevich and Zorich~\cite{KZ} (in the orientable case) and Lanneau~\cite{Lanneau} (in the non-orientable case) gave a complete classification of the connected components of all strata of holomorphic quadratic differentials. Taking this into account, we have the following result:
\begin{theorem}[Masur~\cite{Masur1}, Veech~\cite{Veech}] The total volume of $\mu_\kappa^{(1)}$ is finite and the Teichm\"uller flow $G_t=\textrm{diag}(e^t,e^{-t})$ is ergodic on each connected component of $\mathcal{M}_\kappa$ with respect to $\mu_\kappa^{(1)}$.
\end{theorem}

In order to analyze the Lyapunov spectrum (i.e., the collection of the Lyapunov exponents) of the Teichm\"uller flow, Kontsevich and Zorich~\cite{K} introduced the following notion: the \emph{Kontsevich-Zorich} cocycle $G_t^{KZ}$ is the quotient of the trivial cocycle $G_t\times id: Q_g\times H^1(M,\mathbb{R})\to Q_g\times H^1(M,\mathbb{R})$ with respect to the action of the mapping class group $\Gamma_g$. It is known that the cocycle $G_t^{KZ}$ is \emph{symplectic}, so that the Lyapunov spectrum of $G_t^{KZ}$ with respect to any $G_t$-invariant ergodic probability $\mu$ is \emph{symmetric}:
\begin{equation*}
1=\lambda_1^\mu\geq\dots\lambda_g^\mu\geq 0\geq-\lambda_g^\mu\geq\dots\geq-\lambda_1^\mu=-1.
\end{equation*}
It turns out that the $g$ numbers $1=\lambda_1^\mu\geq \lambda_2^\mu\geq\dots\geq\lambda_g^\mu$ appearing in the non-negative part of the Kontsevich-Zorich spectrum determines the Lyapunov spectrum of the Teichm\"uller flow (this is one of the motivation for introducing the Kontsevich-Zorich cocycle). Indeed, it is possible to show that, for any ergodic probability measure $\mu$ on $\mathcal{M}_g$ supported on a stratum of orientable quadratic differentials with $\sigma\in \N$ 
distinct zeros, the Lyapunov spectrum of the Teichm\"uller flow  is
\begin{equation*}
\begin{aligned}
2&=(1+\lambda_1^\mu)\geq \dots\geq (1+\lambda_g^{\mu})\geq \overbrace{1=\dots=1}^{\sigma-1}\geq
(1-\lambda_g^\mu) \geq \dots \\ &\geq (1-\lambda_2^\mu)\geq 0\geq-(1-\lambda_2^\mu)\geq\dots\geq-(1-\lambda_g^\mu) \\
&\geq \underbrace{-1=\dots=-1}_{\sigma-1} \geq-(1+\lambda_g^\mu)\geq \dots 
\geq-(1+\lambda_1^\mu)=-2.
\end{aligned}
\end{equation*}

On the other hand, concerning the Lyapunov spectrum of the Kontsevich-Zorich cocycle, Zorich and Kontsevich conjectured that the Lyapunov exponents of $G_t^{KZ}$ for the canonical absolutely continuous measure $\mu_\kappa^{(1)}$ on any stratum of orientable quadratic differentials 
are all non-zero (i.e., non-uniform hyperbolicity) and distinct (i.e., all Lyapunov exponents have multiplicity $1$). After the fundamental works of G. Forni~\cite{Forni} (showing the non-uniform hyperbolicity of $G_t^{KZ}$) and Avila, Viana~\cite{AV} (proving the simplicity of the Lyapunov spectrum), it follows that Zorich-Kontsevich conjecture is true. In other words, the Lyapunov exponents of a $\mu_\kappa^{(1)}$-generic point are all non-zero and they have multiplicity $1$. However, it remains to understand the dynamical behavior under $G_t^{KZ}$ of the non-generic orbits (with respect to $\mu_\kappa^{(1)}$). In particular, one can follow Veech and ask how ``degenerate'' the Lyapunov spectrum of $G_t^{KZ}$ can be along a non-typical orbit. This question was first answered by G. Forni~\cite{ForniSurvey} who exhibited an example of an orientable  holomorphic quadratic 
differential $q$ on a Riemann surface $M$ of genus $g=3$ such that the Kontsevich-Zorich
spectrum of the $SL(2,\mathbb{R})$-invariant measure $\mu$ supported on the (closed) $SL(2,\mathbb{R})$-orbit of $(M,q)$ verifies $\lambda_2^\mu=\lambda_3^\mu=0$.

\smallskip
At this point, we are able to state our main result:

\begin{theorem}\label{t.FM}There exists an orientable holomorphic quadratic differential $q$ on a Riemman surface $M$ of genus 4 such that the Lyapunov exponents (with respect to $G_t^{KZ}$) of the $SL(2,\mathbb{R})$-invariant probability $\mu$ supported on the $SL(2,\mathbb{R})$-orbit of $(M,q)$ verifies
$$\lambda_2^{\mu}=\lambda_3^{\mu}=\lambda_4^{\mu}=0.$$
\end{theorem}

The organization of this note is the following: in \S 2, we will recall Forni's version of the Kontsevich-Zorich formula for the sum of the Lyapunov exponents of $G_t^{KZ}$, in \S 3 we review Forni's method
\cite{ForniSurvey} to exlpoit symmetries and, in section \S 4, we will complete the proof of our theorem.

\smallskip
{\bf Acknowledgements:}  We would like to thank J.-C. Yoccoz,  who motivated our work by asking whether the example in \cite{ForniSurvey} could be generalized. We are also very grateful to 
M. M\"oller who introduced us to the notion of a cyclic cover, directed us to the reference \cite{B} and suggested the correct algebraic equation of our example.

\section{The Kontsevich-Zorich formula revisited}

Let $q$ be a holomorphic quadratic differential on a Riemann surface $M$ of genus $g\geq 2$. Fix $z=x+iy$ a holomorphic local coordinate and write $q=\phi(z)dz^2$. It follows that the (degenerate) Riemannian metric $R_q$ and the area form $\omega_q$ induced by $q$ are $R_q=|\phi(z)|^{1/2}(dx^2+dy^2)^{1/2}$, $\omega_q=|\phi(z)|dx\wedge dy$. Denote by $S=\partial/\partial x$ and $T=\partial/\partial y$ the horizontal and vertical directions. Define $L^2_q(M):=L^2(M,\omega_q)$ the space of complex-valued square-integrable functions and $H^1_q(M)$ the Sobolev space of functions $v\in L^2_q(M)$ such that $Sv,Tv\in L^2_q(M)$.

\begin{lemma}[Forni~\cite{Forni1}, prop. 3.2]The Cauchy-Riemann operators $\partial^{\pm}:=\frac{S\pm iT}{2}$ with (dense) domain $H^1_q(M)\subset L^2_q(M)$ are closed. Morever, $\partial^{\pm}$ has closed range of finite codimension (equal to the genus $g$ of $M$). Furthermore, denoting by $\mathfrak{M}_q^{\pm}\subset L^2_q(M)$ the subspaces of meromorphic (resp. anti-meromorphic) functions, we have the orthogonal decompositions:
$$L^2_q(M)=\textrm{Ran}(\partial_q^+)\oplus \mathfrak{M}_q^- = \textrm{Ran}(\partial_q^-)\oplus \mathfrak{M}_q^+.$$
\end{lemma}

Denote by $\pi_q^{\pm}:L^2_q(M)\to\mathfrak{M}_q^{\pm}$ the orthogonal projection. Let $H_q$ be the (non-negative definite) Hermitian form on $\mathfrak{M}_q^+\subset L^2_q(M)$ given by
$$H_q(m_1^+,m_2^+):=(\pi_q^-(m_1^+),\pi_q^-(m_2^+))_q:=\int_M \pi_q^-(m_1^+)\cdot\overline{\pi_q^-(m_2^+)} \omega_q$$
for all $m_1^+,m_2^+\in\mathfrak{M}_q$. Let
$$1\equiv\Lambda_1(q)\geq \Lambda_2(q)\geq\dots\geq\Lambda_g(q)\geq 0$$
be the eigenvalues of $H_q$.

Next, we consider the sets
$$\mathcal{R}_g^{(1)}(k):=\{q\in\mathcal{M}_g^{(1)}: \Lambda_{k+1}(q)=\dots=\Lambda_g(q)=0\}.$$

\begin{definition}The set $\mathcal{R}_g^{(1)}(g-1)$ is called the \emph{determinant locus}.
\end{definition}

The relevance of the natural filtration of sets $\mathcal{R}_g^{(1)}(1)\subset\dots\subset\mathcal{R}_g^{(1)}(g-1)$ for the study of the Lyapunov spectrum of $G_t^{KZ}$ becomes evident from the following version of a formula by Kontsevich and
Zorich \cite{K} for the sum of Lyapunov exponents:

\begin{theorem}[Forni~\cite{Forni}, Corollary 5.3]\label{t.formula}Let $\mu$ be a $SL(2,\mathbb{R})$-invariant ergodic probability measure on $\mathcal{M}_g^{(1)}$. Then, the Lyapunov exponents of $G_t^{KZ}$ with respect to $\mu$ satisfy the formula:
$$\lambda_1^{\mu}+\dots+\lambda_g^\mu = \int_{\mathcal{M}_g^{(1)}}(\Lambda_1(q)+\dots+\Lambda_g(q)) d\mu(q).$$
In particular, since $\lambda_1^\mu=1\equiv\Lambda_1(q)$, we have
$$\lambda_2^{\mu}+\dots+\lambda_g^\mu = \int_{\mathcal{M}_g^{(1)}}(\Lambda_2(q)+\dots+\Lambda_g(q)) d\mu(q).$$
\end{theorem}

A direct consequence of this formula is:
\begin{corollary}[Forni~\cite{ForniSurvey}, Corollary 7.1]
\label{c}
Let $\mu$  be any $SL(2,\mathbb{R})$-invariant ergodic probability measure  on $\mathcal{M}_g^{(1)}$
supported on a stratum of orientable quadratic differentials.  The measure $\mu$ is supported on the locus $\mathcal{R}_g^{(1)}(1)$ if and only if the non-trivial Kontsevich-Zorich spectrum vanishes, that is,
$$
\lambda_2^\mu=\dots=\lambda_g^\mu=0 \,.
$$
\end{corollary}

\section{Symmetries}
In this section we recall the simple method developed in \cite{ForniSurvey} to derive bounds
on the rank of the matrix $H_q$ from {\it symmetriesÊ}of the orientable holomorphic 
quadratic differential $q$. Let $\text{ \rm Aut}(M)$ be the goup of holomorphic automorphism of 
the Riemann surface $M$. Let $\text{ \rm Aut}(q)\subset \text{ \rm Aut}(M)$ is the subgroup formed by automorphisms $a\in \text{ \rm Aut}(M)$ such that $a^*(q)=q$. Note that there is a natural unitary action of $\text{ \rm Aut}(q)$ on $\mathfrak{M}_q^+$ (by pull-back). Given $a\in \text{ \rm Aut}(q)$, fix $\{m_1^+(a),\dots, m_g^+(a)\}$ an orthonormal basis of eigenvectors and denote by $\{u_1(a),\dots,u_g(a)\}$ the associated eigenvalues. Let $B^a(q)$ be the matrix of the operator $\pi_q^-:\mathfrak{M}_q^+\to\mathfrak{M}_q^-$ with respect to the basis $\{m_1^+(a),\dots,m_g^+(a)\}\subset\mathfrak{M}_q^+$ and $\{\overline{m_1^+(a)},\dots,\overline{m_g^+(a)}\}\subset\mathfrak{M}_q^-$:
$$B^a_{ij}(q) = \int_M m_i^+(a)m_j^+(a)\omega_q.$$
For any $I,J\subset\{1,\dots,g\}$ with $\#I=\#J$, denote by $\det B_{IJ}^a$ the minor of the matrix $B^a(q)$ with entries $B_{ij}^a(q)$, $i\in I, j\in J$.

\begin{lemma}[Forni~\cite{ForniSurvey}, Lemma 7.2] The following holds:
$$\prod\limits_{i\in I}\prod\limits_{j\in J} u_i(a)u_j(a)\neq 1 \Longrightarrow \det B_{IJ}^a(q)=0.$$
\end{lemma}

\begin{corollary}\label{c1} If $\prod\limits_{i\in I}\prod\limits_{j\in J} u_i(a)u_j(a)\neq 1$
for all $I,J\subset\{1,\dots,g\}$ with $\#I=\#J=k$, then $\textrm{rank}(H_q)\leq g-k$, hence
$q\in \mathcal{R}_g^{(1)}(k)$.
\end{corollary}
\begin{proof} Since the matrix $H(q)$ of $H_q$ with respect to $\{m_1^+(a),\dots,m_g^+(a)\}$ satisfies $H(q)=B^a(q)^*B^a(q)=\overline{B^a(q)} B^a(q)$ (see equation (44) of~\cite{ForniSurvey}), the 
statement follows immediately from the previous lemma.
\end{proof}

At this point, we are ready to present our genus $4$ example and prove Theorem~\ref{t.FM}.

\section{The example}

Given an integer $N>1$ and a $4$-tuple of integers $\textbf{a}:=(a_1,\dots,a_4)$ with $0<a_{\mu}<N$,
$\text{gcd}(N, a_1, \dots,a_4) =1$  and $\sum\limits_{\mu=1}^4 a_{\mu}\equiv 0 \, (\textrm{mod } N)$.
Let $x_1,x_2,x_3,x_4\in \C$ be $4$ distinct points and let $M:=M_N(\bf{a})$ be the connected, non-singular Riemann surface determined by the algebraic equation:
$$w^N=(z-x_1)^{a_1}(z-x_2)^{a_2}(z-x_3)^{a_3}(z-x_4)^{a_4}\,.$$
The surface $M$ is a {\it cyclic cover } of the Riemann sphere $\Proj^1(\C)$ branched over the 
 points $x_1, \dots, x_4 \in \C \equiv \Proj^1(\C)\setminus\{\infty\}$. In fact, every such cyclic cover 
 (up to isomorphisms) can be written as above (see Example 4.2 in \cite{B}). The surface
 $M$ has genus 
$$
g=N+1- \frac{1}{2} \sum\limits_{\mu=1}^4\textrm{gcd}(a_\mu,N)\,.
$$
The surface $M$ is called a cyclic cover since its automorphism group is cyclic. In fact, it
is generated by the automorphism $T:M\to M$ given by
$$T(z,w)=(z,\varepsilon w) \,,$$
where $\varepsilon^N=1$ is a primitive $N$th root of unity.

For $i\in \{1, \dots, N-1\}$, let $L_i$ denote the eigenspace with eigenvalue $\epsilon^i$ for the
action of $T$ on holomorphic abelian differentials: 
$$
L_i:=\{\omega\in H^1_{\textrm{dR}}(M,\C): T^{*}\omega=\varepsilon^i\cdot\omega\}\,.
$$

\begin{lemma} [Bouw \cite{B}, Lemma 4.3]
\label{l.Bouw} The following formulas hold:
$$
\textrm{dim}_{\mathbb{C}}L_i=\sum\limits_{\mu=1}^4\langle \frac{i\cdot a_{\mu}}{N}\rangle \,\,-\,\,1
$$
(here $\langle\cdot \rangle$ denotes the fractional part).
\end{lemma}

We recall that the genus $3$ example in \cite{ForniSurvey} was given by the $SL(2,\R)$ orbit of  orientable  quadratic differentials (or, equivalently, of abelian differentials) 
constructed from cyclic covers of type $M_4(1,1,1,1)$.  

\begin{proof}  [Proof of Theorem \ref{t.FM}] 
Our example is given by an orientable holomorphic quadratic differential 
constructed from cyclic covers of type $M_6(1,1,1,3)$. 

Let $M$ be any branched cover of type $M_6(1,1,1,3)$. Observe that $M$  has genus $4$. Lemma~\ref{l.Bouw} yields: 
\begin{equation}
\begin{aligned}
\textrm{dim}_{\mathbb{C}}L_1&=\textrm{dim}_{\mathbb{C}}L_2=0 \,; \\
\textrm{dim}_{\mathbb{C}}L_3&=\textrm{dim}_{\mathbb{C}}L_4=1 \,;\\
 \textrm{dim}_{\mathbb{C}}L_5&=2\,.
 \end{aligned}
 \end{equation}
Thus, we can fix $\theta_1\in L_3$, $\theta_2\in L_4$ and $\theta_3,\theta_4\in L_5$ so that $\{\theta_1,\dots,\theta_4\}$ is a basis of the space of holomorphic differentials on $M$ and
the action  of $T^*$ on this basis is diagonal with eigenvalues $-1=\varepsilon^3$ (with multiplicity $1$), $\varepsilon^4$ (with multiplicity $1$) and $\varepsilon^5$ (with multiplicity $2$). Let  
$q=\theta_1^2$ be the orientable quadratic differential obtained as pull-back to $M$ of the unique holomorphic quadratic differential on $\Proj(\C)$ with simple poles at the branching points 
$\{x_1,\dots, x_4\}$. It follows that the spectrum of the action of $T\in \text{ \rm Aut}(q)$ on the space 
$\mathfrak{M}_q^+\subset L_q^2(M)$ is
$$u_1(T)=1\,, \quad u_2(T)=-\varepsilon^4\,, \quad u_3(T)=-\varepsilon^5\,,
 \quad u_4(T)=-\varepsilon^5\,.$$
Consequently, $q\in\mathcal{R}_4^{(1)}(1)$. In fact, a direct application of Corollary~\ref{c1} shows that the form $H_q$ has rank $1$, so that $q\in\mathcal{R}_4^{(1)}(1)$. 

Let $\Cal V$ be the set of all orientable quadratic differentials on a cyclic cover of type 
$M_6(1,1,1,3)$ branched over $4$ points $\{x_1, \dots, x_4\}\in \Proj(\C)$ obtained as
pull-back of the unique (non-orientable) holomorphic quadratic  differential on $\Proj(\C)$, 
of total area equal to $1$, with simple poles at  $x_1, \dots, x_4$. Since the stratum of all quadratic differentials on the Riemann sphere with $4$ simple poles is $SL(2,\R)$-invariant and consists of a single $SL(2,\R)$ orbits, the set $\Cal V$ consists of a single $SL(2,\R)$ orbit in the stratum of orientable quadratic differentials with $3$ distinct zeros of order $4$ (equivalently, of the stratum of squares of abelian differentials with $3$ distinct double zeros). The above computation shows that $\Cal V \subset \mathcal{R}_4^{(1)}(1)$, hence (in view of Corollary~\ref{c}) the proof of Theorem~\ref{t.FM} is completed.
\end{proof}

We conclude by the remark that our exploration of other cyclic covers has failed to
provide additional example of totally degenerate Kontsevich-Zorich spectra. We recall
that according to M\"oller \cite{Moeller}, there are no more examples of totally degenerate
Kontsevich-Zorich spectrum among Veech surfaces in any genus.

\end{document}